\documentclass{elsarticle}

\usepackage{amsmath,amssymb,amsthm,mathrsfs}
\usepackage{tikz}
\usetikzlibrary{positioning}
\usepackage{color}

\theoremstyle{plain}
\numberwithin{equation}{section}
\theoremstyle{plain}
\newtheorem{theorem}{Theorem}[section]
\newtheorem{corollary}[theorem]{Corollary}
\newtheorem{lemma}[theorem]{Lemma}
\newtheorem{proposition}[theorem]{Proposition}
\theoremstyle{definition}

\newcommand{\ands}{\quad\mbox{and}\quad}
\newcommand{\kernel}{\operatorname{Ker}}
\newcommand{\degr}{\operatorname{deg}}
\newcommand{\Dom}{\operatorname{Dom}}
\newcommand{\codim}{\operatorname{codim}}

\newcommand{\Rat}{{\mathrm{Rat}}}
\newcommand{\Ran}{\operatorname{Ran}}
\newcommand{\diag}{\operatorname{Diag}}
\newcommand{\Row}{\operatorname{Row}}
\newcommand{\Col}{\operatorname{Col}}

\newcommand{\BC}{{\mathbb C}}
\newcommand{\BT}{{\mathbb T}}

\newcommand{\BD}{{\mathbb D}}

\newcommand{\BP}{{\mathbb P}}

\newcommand{\al}{\alpha}
\newcommand{\be}{\beta}
\newcommand{\ga}{\gamma}
\newcommand{\de}{\delta}

\newcommand{\Th}{\Theta}

\newcommand{\la}{\lambda}

\newcommand{\Up}{\Upsilon}

\newcommand{\Om}{\Omega}
\newcommand{\fA}{{\mathfrak A}}

\newcommand{\cC}{{\mathcal C}}\newcommand{\cD}{{\mathcal D}}

\newcommand{\cO}{{\mathcal O}}\newcommand{\cP}{{\mathcal P}}

\newcommand{\cX}{{\mathcal X}}

\newcommand{\ov}[1]{{\overline{#1}}}

\newcommand{\mat}[1]{\ensuremath{\begin{bmatrix} #1 \end{bmatrix}}}

\begin{document}
\begin{frontmatter}
\title{A Toeplitz-like operator with rational matrix symbol having poles on the unit circle:\ Invertibility and Riccati equations}
\author[G.J. Groenewald]{G.J. Groenewald}\ead{Gilbert.Groenewald@nwu.ac.za}
\author[S. ter Horst]{S. ter Horst}\ead{Sanne.TerHorst@nwu.ac.za}
\author[J. Jaftha]{J. Jaftha}\ead{Jacob.Jaftha@uct.ac.za}
\author[A.C.M. Ran]{A.C.M. Ran}\ead{a.c.m.ran@vu.nl}

\address[G.J. Groenewald]{School of Mathematical and Statistical Sciences,
North-West University,
Research Focus: Pure and Applied Analytics,
Private~Bag X6001,
Potchefstroom 2520,
South Africa.}

\address[S. ter Horst]{School of Mathematical and Statistical Sciences,
North-West University,
Research Focus: Pure and Applied Analytics,
Private~Bag X6001,
Potchefstroom 2520,
South Africa
and
DSI-NRF Centre of Excellence in Mathematical and Statistical Sciences (CoE-MaSS),
Johannesburg,
South Africa}

\address[J. Jaftha]{Numeracy Center, University of Cape Town, Rondebosch 7701, Cape Town, South Africa}

\address[A.C.M. Ran]{Department of Mathematics, Faculty of Science, VU Amsterdam, De Boelelaan 1111, 1081 HV Amsterdam, The Netherlands and Research Focus: Pure and Applied Analytics, North-West University, Potchefstroom, South Africa}


\begin{abstract}
This paper is a continuation of the work on unbounded Toeplitz-like operators $T_\Om$ with rational matrix symbol $\Om$ initiated in Groenewald et.\ al (Complex Anal.\ Oper.\ Theory 15, 1(2021)), where a Wiener-Hopf type factorization of $\Om$ is obtained and used to determine when $T_\Om$ is Fredholm and compute the Fredholm index in case $T_\Om$ is Fredholm. Due to the high level of non-uniqueness and complicated form of the Wiener-Hopf type factorization, it does not appear useful in determining when $T_\Om$ is invertible. In the present paper we use state space methods to characterize invertibility of $T_\Om$ in terms of the existence of a stabilizing solution of an associated nonsymmetric discrete algebraic Riccati equation, which in turn leads to a pseudo-canonical factorization of $\Om$ and concrete formulas of $T_\Om^{-1}$.
\end{abstract}

\begin{keyword}
Toeplitz operators, unbounded operators, invertibility, Riccati equations, pseudo-canonical factorization

\emph{2020 MSC} Primary 47B35, 47A53; Secondary 47A68
\end{keyword}

\end{frontmatter}

{\date{}}


\section{Introduction}

In two recent papers \cite{GtHJR21,GtHJRsub}, we explored the matrix analogue of an unbounded Toeplitz-like operator that was investigated in \cite{GtHJ18,GtHJ19a,GtHJ19b} for scalar rational symbols with poles on the unit circle $\BT$. While many of the classical operator theory topics, like Fredholmness, invertibility and spectrum, are well understood in the scalar case, the case of matrix symbols appears to be more intricate. In \cite{GtHJR21} a Wiener-Hopf type factorization was obtained, from which the Fredholm index can be determined, in case the symbol has no zeroes on $\BT$. However, this Wiener-Hopf type factorization has a high level of non-uniqueness and, unlike in the classical case, generally does not lead to a diagonalization of the symbol. As a result, although some further Fredholm characteristics can be determined from the Wiener-Hopf type factorization \cite{GtHJRsub}, it does not seem to be an adequate tool to compute the dimensions of the kernel and the cokernel, nor does it seem to give a clear characterization of invertibility. In the present paper we take a different approach to the question of invertibility of this Toeplitz-like operator, using Riccati equations and pseudo-canonical factorization.

Unbounded Toeplitz operators appeared first in a paper of Hartman and Wintner \cite{HW50} in 1950, but only became an active topic with the seminal paper of Sarason \cite{S08} in connection to truncated Toeplitz operators; see \cite{O'L22} for how unbounded Toeplitz operators with matrix symbols come into play. More recently, kernels of unbounded Toeplitz operators appeared in the study of nearly backward shift invariant subspaces and Toeplitz inverses \cite{CMP21,CPpre}.

Next we introduce some notation, which is required to define our Toeplitz-like operator and state our main results. We write $\Rat$ for the space of rational functions, $\Rat(\BT)$ for the functions in $\Rat$ that only have poles on the unit circle $\BT$, $\Rat_0(\BT)$ for the strictly proper functions in $\Rat(\BT)$, $\cP$ for the space of polynomials and for positive integers $k$ we indicate the polynomials of degree at most $k$ by $\cP_k$, extending it to all integers by setting $\cP_k=\{0\}$ in case $k\leq 0$. For positive integers $m$ and $n$, we indicate the spaces of $m \times n$ matrices with entries from these function spaces by $\Rat^{m\times n}$, $\Rat_0(\BT) ^{m\times n}$, etc. In the case of vector functions, when $n=1$, we will just write $m$ instead of $m \times 1$. For $1<p<\infty$, $L^p$ and $H^p$ denote the Lebesgue space and Hardy space, respectively, and $K^p$ is the standard complement of $H^p$ in $L^p$. With $L_m^p$,  $H_m^p$ and $K_m^p$ we indicate the spaces of vectors of length $m$ with entries from $L^p$ and $H^p$, respectively.

Let $\Om \in\Rat^{m\times  m}$ with possibly poles on $\BT$ and $\det\Om\not \equiv 0$, and let $1<p<\infty$. We then define the Toeplitz-like operator $T_\Om \left (H^p_m \rightarrow H^p_m \right )$ by
\begin{equation}\label{TOm}
\begin{aligned}
\Dom(T_\Om)=\left\{\begin{array}{ll}
f\in H^p_m :&
\begin{array}{l}
\Om f = h + \eta \textrm{ where } h\in L^p_m(\BT),\\
\ands \eta \in\Rat_0^{m}(\BT) \\
\end{array}\\
\end{array}
\right \}, \\
T_\Om f = \BP h \textrm{ with } \BP \textrm{ the Riesz projection of } L^p_m(\BT) \textrm{ onto }H^p_m.
\end{aligned}
\end{equation}
By the Riesz projection, $\BP$, we mean the projection of $L^p_m$ onto $H^p_m$, as discussed in \cite[pages 149--153]{H62}.

As usual, $\Om \in\Rat^{m\times  m}$ has a pole at $z_0\in\BC\cup \{\infty\}$, if any of its entries has a pole at $z_0$. In case $\det \Om\not\equiv 0$, a zero of $\Om$ is a pole of its inverse $\Om^{-1}(z):=\Om(z)^{-1}$. It is not necessarily the case that the zeroes of $\Om$ correspond to the zeroes of $\det \Om$, and $\Om$ can have both a pole and a zero at the same point $z_0\in\BC\cup \{\infty\}$. It was proved in \cite{GtHJR21} that $T_\Om$ is Fredholm if and only if $\Om$ has no zeroes on $\BT$. In particular, for $T_\Om$ to be invertible, that is, bijective from $\Dom(T_{\Om})$ onto $H^p_m$, it is necessary that $\Om$ has no zeroes on $\BT$.

In the classical case, invertibility of Toeplitz operators with rational symbols can be studied via Riccati equations and canonical factorizations associated with state space realizations of the symbol, cf., \cite{BGKR08,BGKR10}. In the present paper, we follow the approach of \cite{FKR10}. Let $\Om\in\Rat^{m \times m}$ and assume $\Om$ is given by a minimal state space realization of the form
\begin{equation}\label{realOmega}
\Om(z)=R_0+zC(I-zA)^{-1}B+\gamma(z I-\alpha)^{-1}\beta,
\end{equation}
with $R_0\in\BC^{m \times m}$ and with $A,B,C$ and $\al,\be,\ga$ matrices of appropriate sizes such that $A$ has all its eigenvalues in the open unit disc $\BD$ and $\al$ has all its eigenvalues in the closed unit disc $\ov{\BD}$, that is, $A$ is stable and $\al$ is semi-stable (in \cite{FKR10} $\al$ is also stable). Minimality in this setting means that one cannot find a representation for $\Om$ of this form with $A$ and $\al$ matrices of smaller size; equivalently, the triples $(C,A,B)$ and $(\ga,\al,\be)$ both provides observable and controllable discrete-time linear systems, cf., \cite{HRvS07}. Despite the fact that $\Om$ has poles on $\BT$, so that $\al$ has eigenvalues on $\BT$, there is a fairly direct analogue of the canonical factorization result of \cite{FKR10}.

\begin{theorem}\label{T:Main1}
Let $\Om\in\Rat^{m \times m}$ with $\det \Om \not\equiv 0$ and no zeroes on $\BT$ and assume that $\Om$ is given by the minimal realization \eqref{realOmega} with $A$ stable and $\al$ semi-stable. Then the following are equivalent:
\begin{itemize}
\item[(i)] There exists a matrix $Q$ such that $R_0-\gamma QB$ is invertible, $Q$ satisfies the nonsymmetric discrete algebraic Riccati equation:
\begin{equation}\label{Ricc}
Q=\alpha QA+\left(\beta-\alpha QB\right)(R_0-\gamma QB)^{-1} (C-\gamma QA),
\end{equation}
and  such that the matrices
\begin{equation}\label{Acirc-alphacirc}
\begin{aligned}
A_{\circ}&:=A-B(R_0-\gamma QB)^{-1}(C-\gamma QA),\\
\alpha_{\circ}&:=\alpha -\left(\beta-\alpha QB\right)(R_0-\gamma QB)^{-1}\ga
\end{aligned}
\end{equation}
are both stable.

\item[(ii)]
$\Omega$ has a right pseudo-canonical factorization $\Omega(z)=\Psi(z)\Theta(z)$, i.e., $\Psi$ and $\Theta$  are $m\times m$ rational matrix functions with $\det\Psi\not\equiv 0$ and $\det\Theta \not\equiv 0$ and such that $\Theta$ has poles only outside or on the unit circle $\BT$, $\Theta^{-1}$ has poles only outside $\BT$, $\Psi$ has poles only inside or on the unit circle $\BT$, and $\Psi^{-1}$ has poles only inside the unit circle $\BT$.

\end{itemize}
Moreover, if $Q$ is a solution to \eqref{Ricc} such that $A_{\circ}$ and $\al_\circ$ are stable, then a right pseudo-canonical factorization is obtained as follows:\ Let $\delta$ and $D$ be invertible matrices such that $\delta D=R_0-\gamma QB$ and set
\begin{equation}\label{Ccirc-betacirc}
C_{\circ}=\delta^{-1}(C-\gamma QA) \ands
\beta_{\circ}=(\beta-\alpha QB)D^{-1}.
\end{equation}
Then $\Omega(z)=\Psi(z)\Theta(z)$ holds with $\Psi$ and $\Theta$ defined as
\begin{equation}\label{ThetaPsi}
\Theta(z)=D+zC_\circ(I-zA)^{-1}B \ands \Psi(z)=\delta+\gamma(zI-\alpha)^{-1}\beta_{\circ},
\end{equation}
and the inverses of $\Psi$ and $\Theta$ are given by
\begin{equation}\label{ThetaPsiInvReal}
\begin{aligned}
\Theta^{-1}(z)& =D^{-1}-zD^{-1}C_{\circ}(I-zA_\circ)^{-1}BD^{-1},\\
\Psi^{-1}(z)& = \delta^{-1}-\delta^{-1}\gamma (z I-\alpha_\circ)^{-1}\beta_{\circ} \delta^{-1}.
\end{aligned}
\end{equation}
Furthermore, the solution $Q$ of the Riccati equation \eqref{Ricc} so that $A_\circ$ and $\al_\circ$ are stable is unique. Finally, the realizations in \eqref{ThetaPsi} and \eqref{ThetaPsiInvReal} for $\Theta$, $\Psi$, $\Theta^{-1}$ and $\Psi^{-1}$ are minimal.
\end{theorem}

The above result is proved in Section \ref{S:FormsForOmr} and is essentially obtained by specifying Theorem 1.1 of \cite{FKR10} for the function $\Om_r$ defined by
\begin{equation}\label{Omr}
\Om_r(z)=\Om(rz)
\end{equation}
for $1<r$ small enough, so that $\Om_r$ does not have poles and zeroes on $\BT$. More generally, for any function $f$, scalar-, vector- or matrix-valued, and scalar $r>0$ we write $f_r$ for the function $f_r(z)=f(rz)$ defined for $z\in\BC$ for which $rz$ is in the domain of $f$.

In order to characterize invertibility of $T_\Om$ more is required than what is in \cite{FKR10}, we want to compare invertibility of the unbounded operator $T_\Om$ with invertibility of the bounded Toeplitz operator $T_{\Om_r}$. Note that invertibility in both cases means bijectivity on its domain of definition. Hence, the inverse of $T_\Om$ will be bounded. For $r>1$, we define the annulus
\begin{equation}\label{annulus}
\fA_r:=\{z\in\BC \colon r^{-1}<|z|<r\}.
\end{equation}
It turns out that $T_\Om$ and $T_{\Om_r}$ can be compared, not only with respect to invertibility, but even with respect to their Fredholm properties, in case they are Fredholm.

\begin{theorem}\label{T:Main2}
Let $\Om\in\Rat^{m \times m}$ with $\det \Om \not\equiv 0$ and assume that $\Om$ has no zeroes on $\BT$. Define $\Om_r$ by \eqref{Omr}. Let $r_0>1$ be such that $\Om$ has no zeroes in the annulus $\fA_{r_0}$ and no poles in $\fA_{r_0}\setminus\BT$. Then $T_\Om$ is Fredholm and for each $1<r<r_0$, $T_{\Om_r}$ is bounded and Fredholm and we have
\[
\dim \kernel T_\Om= \dim \kernel T_{\Om_r} \ands \codim \Ran T_\Om= \codim \Ran T_{\Om_r}.
\]
In particular, $T_\Om$ is invertible if and only if $T_{\Om_r}$ is invertible for some (and hence all) $1<r<r_0$.
\end{theorem}

We shall prove Theorem \ref{T:Main2} in Section \ref{S:FredholmCompare}. Our second main result together with \cite{FKR10}, shows that invertibility of $T_\Om$ is equivalent to items (i) and (ii) in Theorem \ref{T:Main1}. With some further work we derive, in Section \ref{S:TOm invert} below, formulas for the inverse of $T_\Om$, as given in the following result.

\begin{theorem}\label{T:Main3}
Let $\Om\in\Rat^{m \times m}$ with $\det \Om \not\equiv 0$ and no zeroes on $\BT$ and assume that $\Om$ is given by the minimal realization \eqref{realOmega} with $A$ stable and $\al$ semi-stable. Then $T_\Om$ is invertible if and only if the Riccati equation \eqref{Ricc} has a solution $Q$ such that $A_\circ$ and $\al_\circ$ in \eqref{Acirc-alphacirc} are stable, or, equivalently, if $\Om$ has a pseudo-canonical factorization $\Om(z)=\Psi(z)\Theta(z)$ as in item (ii) of Theorem \ref{T:Main1}. In that case, the inverse of $T_{\Om}$ is the bounded operator given by
\begin{equation}\label{TOminv}
T_{\Om}^{-1}=T_{\Theta}^{-1}T_{\Psi}^{-1}=T_{\Theta^{-1}}T_{\Psi^{-1}}.
\end{equation}
Moreover, $T_{\Om}^{-1}$ has a block matrix representation $\left[T_{\Om}^{-1}\right]_{i,j}$ with respect to the standard block basis of $H^p_m$ that is given by
\begin{equation}\label{TOminvmat}
\left[T_{\Om}^{-1}\right]_{i,j}=\sum_{k=0}^{\min(i,j)} \Theta_{i-k}^\times \Psi_{j-k}^\times,
\end{equation}
where
\begin{align*}
\Theta_{0}^\times&=D^{-1}, &\Theta_{j}^\times =-D^{-1}(C_\circ)(A_\circ)^{j-1}BD^{-1},\ \   j=1,2,\ldots,\\
\Psi_{0}^\times&=\delta^{-1}, &\Psi_{j}^\times =-\delta^{-1}\gamma \left(\alpha_\circ\right)^{j-1}\left(\beta_\circ\right)\delta^{-1},\ \  j=1,2,\ldots.
\end{align*}
\end{theorem}

For the final result we present in this introduction we restrict to the case where $p=2$, since it relies on a result from \cite{FKR10}, which is proved only for $p=2$.  Let $\Om\in\Rat^{m \times m}$ be given by the minimal realization with $A$ stable and $\al$ semi-stable and assume  $A$ and $\al$ are of size $s \times s$ and $t\times t$, respectively. We define the observability operator for the pair $(C,A)$ as
\begin{equation}\label{ObsCA}
\cO_{C,A}:\BC^s \to H^2_m,\quad \cO_{C,A}:x\mapsto C(I-zA)^{-1}x,\ \  \la\in\BD,
\end{equation}
and the controllability operator for the pair $(\al,\be)$ as
\begin{equation}\label{Conalbe}
\begin{aligned}
&\cC_{\al,\be}:\Dom(\cC_{\al,\be})\to \BC^t,\quad \cC_{\al,\be} f= \frac{1}{2\pi}\int_{-\pi}^{\pi} (e^{it}I-\al)^{-1}\be f(e^{it})\, \textup{d}t,\\
&\mbox{for}\quad f\in \Dom(\cC_{\al,\be}):=\left\{ f\in H^2_m \colon \int_{-\pi}^{\pi} (e^{it}I-\al)^{-1}\be f(e^{it})\, \textup{d}t\mbox{ exists} \right\}.
\end{aligned}
\end{equation}
Since $A$ is stable, it is clear that $\cO_{C,A}$ defines a bounded operator from $\BC^s$  into $H^2_m$. Due to the semi-stability of $\al$, $\cC_{\al,\be}$ need not be bounded, but it is the case that the subspace
\[
\cD_r:=\{f\in H^2_m \colon f_r\in H^2_m\},
\]
for $r>1$, is contained in $\Dom(\cC_{\al,\be})$. This will be proved in Lemma \ref{L:rtononr} in Section \ref{S:TOm invert} below, where we will also prove the following proposition.

\begin{proposition}\label{P:Qform}
Consider the case $p=2$. Let $\Om\in\Rat^{m \times m}$ with $\det \Om \not\equiv 0$ and such that $T_\Om$ is invertible. Assume that $\Om$ is given by the minimal realization \eqref{realOmega} with $A$ stable and $\al$ semi-stable. Then the solution $Q$ of the algebraic Riccati equation \eqref{Ricc} that makes $A_\circ$ and $\al_\circ$ in \eqref{Acirc-alphacirc} stable is given by
\begin{equation}\label{Qform}
Q=\cC_{\al,\be} T_\Om^{-1} \cO_{C,A}.
\end{equation}
\end{proposition}

The formula for $Q$ is analogous to that in \cite{FKR10}, where poles on $\BT$ are not allowed, but requires more attention since $\cC_{\al,\be}$ is not necessarily bounded. To see that the right hand side in \eqref{Qform} is well defined, we point out that $\cO_{C,A}$ maps $\BC^s$ into $\cD_r$ for $r>1$ small enough, while $T_\Om^{-1}$ maps $\cD_r$  into $\cD_r$, again for $r>1$ small enough, which is contained in the domain of $\cC_{\al,\be}$. That $T_\Om^{-1}$ maps $\cD_r$  into $\cD_r$ follows from Proposition \ref{P:TomVsTomr} below.

We conclude this introduction with a brief overview of the remainder of the paper. In Section \ref{S:FormsForOmr} we apply the main result of \cite{FKR10} to the function $\Om_r$ in \eqref{Omr}, and translate back to the state space realization of $\Om$, leading to a proof of Theorem \ref{T:Main1}. In the next section we investigate the relation between $T_\Om$ and $T_{\Om_r}$, and give a proof of Theorem \ref{T:Main2}. The work of Sections \ref{S:FormsForOmr} and \ref{S:FredholmCompare}, is then combined in Section \ref{S:TOm invert} to prove Theorem \ref{T:Main3} as well as Proposition \ref{P:Qform}.

\section{Riccati equation, canonical factorization and 
inversion for $\Om_r$}\label{S:FormsForOmr}

Suppose that $\Om\in\Rat^{m \times m}$ is given by the minimal realization formula \eqref{realOmega}, that is:
\begin{equation}\label{RealOm}
\Om(z)=R_0+zC(I-zA)^{-1}B+\gamma(z I-\alpha)^{-1}\beta,
\end{equation}
with $A$ being stable and $\al$ being semistable. It is then clear that $\Om_r$ defined by \eqref{Omr} admits the state space realization
\begin{align}
\Om_r(z)=\Om(rz) &=R_0+z(rC)(I-z(rA))^{-1}B+\gamma\left(z I-\frac{\alpha}{r}\right)^{-1}\frac{\beta}{r} \notag\\
&=R_0+zC_r(I-zA_r)^{-1}B+\gamma(z I-\al_r)^{-1}\be_r \label{realOmr}
\end{align}
with
\begin{equation}\label{OmrRealMatrices}
A_r=rA,\quad C_r=rC,\quad \al_r= \frac{\alpha}{r},\quad \be_r =\frac{\beta}{r}.
\end{equation}
As in Theorem \ref{T:Main1}, let $r_0>1$ be such that $\Om$ has no zeroes in the annulus $\fA_{r_0}$ and no poles in $\fA_{r_0}\setminus\BT$, with $\fA_{r_0}$ as defined in \eqref{annulus}. Then, for $1<r<r_0$, $\Om_r$ has no poles and no zeroes on $\BT$. Therefore, the results of \cite{FKR10} apply to $\Om_r$ and its realization \eqref{realOmr}--\eqref{OmrRealMatrices}. Note that the paper \cite{FKR10} only considers the case of Toeplitz operators on $H^2_m$. However, since invertibility of Toeplitz operators on $H^p_m$ with rational matrix symbols can be characterized in terms of their Wiener-Hopf factorizations, which are independent of $p$, invertibility on $H^p_m$ is independent of the value of $p$. We now specify the main result of \cite{FKR10} to $\Om_r$, together with some supplementary observations, in the next proposition. This result is subsequently used to prove Theorem \ref{T:Main1}.

\begin{proposition}\label{P:OmrThm}
Let $\Om\in\Rat^{m \times m}$ be given by the realization \eqref{RealOm} with $A$ stable and $\al$ semi-stable, so that $\Om_r$ is given by the realization \eqref{realOmr}--\eqref{OmrRealMatrices}.  Let $r_0>1$ be such that $\Om$ has no zeroes in the annulus $\fA_{r_0}$ and no poles in $\fA_{r_0}\setminus\BT$. For $1<r<r_0$ the following are equivalent:
\begin{itemize}
\item[(i)] $T_{\Omega_r}$ is invertible.

\item[(ii)] There exists a matrix $Q$ such that $R_0-\gamma QB$ is invertible, $Q$ satisfies the  nonsymmetric discrete algebraic Riccati equation:
\begin{equation}\label{RiccOmr}
Q=\alpha QA+\left(\beta-\alpha QB\right)(R_0-\gamma QB)^{-1} (C-\gamma QA),
\end{equation}
and $r A_\circ$ and $r^{-1} \al_\circ$ are stable, with $A_\circ$ and $\al_\circ$ given by \eqref{Acirc-alphacirc}.

\item[(iii)]
$\Omega_r$ has a canonical factorization $\Omega_r(z)=\Psi^{(r)}(z)\Theta^{(r)}(z)$, i.e., $\Psi^{(r)}$ and $\Theta^{(r)}$ are $m \times m$ rational matrix functions with $\det\Psi^{(r)}\not\equiv 0$ and $\det\Theta^{(r)}\not\equiv 0$ and such that $\Theta^{(r)}$ and $(\Theta^{(r)})^{-1}$ have poles only outside $\BT$ and $\Psi^{(r)}$ and $(\Psi^{(r)})^{-1}$ have poles only inside $\BT$.

\end{itemize}
Moreover, the solution $Q$ of the Riccati equation \eqref{RiccOmr} such that $r A_\circ$ and $r^{-1} \al_\circ$ are stable is unique and independent of $r$, i.e., for each $1<r<r_0$ one obtains the same solution $Q$ in item \textup{(ii)}. Furthermore, if $Q$ is as in (ii), then a canonical factorization of $\Om_r$ is obtained with $\Theta^{(r)}=\Theta_r$ and $\Psi^{(r)}=\Psi_r$, where $\Theta$ and $\Psi$ are defined as in Theorem \ref{T:Main1}, and $\Theta_r$ and $\Psi_r$ are defined according to \eqref{Omr}. In case one of items \textup{(i)}--\textup{(iii)} holds, and hence all, we have $T_{\Om_r}^{-1}=T_{\Theta_r}^{-1}T_{\Psi_r}^{-1}=T_{\Theta_r^{-1}}T_{\Psi_r^{-1}}$.
\end{proposition}

\begin{proof}[\bf Proof]
Since $\Om_r$ has no poles on $\BT$, since $1<r<r_0$, Theorem 1.1 of \cite{FKR10} applies to $\Om_r$ and its realization \eqref{realOmr}, leading to the equivalence of variations of (i)--(iii) in terms of the matrices in the realizations. Technically, Theorem 1.1 of \cite{FKR10} does not contain an item about the invertibility of $T_{\Om_r}$, but that invertibility of $T_{\Om_r}$ is equivalent to the two items in the theorem follows by the discussion preceding the theorem, and this is also where the formula for $T_{\Om_r}^{-1}$ in terms of the canonical factors appears. It thus remains to show that the statements of items (ii) and (iii) in terms of the realization matrices of $\Om_r$ correspond to the statements concerning the Riccati solutions and canonical factorization from Theorem 1.1 of \cite{FKR10}, respectively.

We start with item (ii). From Theorem 1.1 of \cite{FKR10}, and the preceding paragraphs, we obtain that invertibility of $T_{\Om_r}$ is equivalent to the existence  of a matrix $Q_r$ such that $R_0-\gamma Q_rB$ is invertible, that satisfies the Riccati equation
\begin{align*}
Q_r&=\alpha_r Q_r A_r+\left(\beta_r-\alpha_r Q_r B\right)(R_0-\gamma Q_r B)^{-1} (C_r-\gamma Q_rA_r)\\
&=\left(\frac{\alpha}{r}\right)Q_r(rA)+\left(\frac{\beta}{r}-\frac{\alpha}{r}Q_rB\right)(R_0-\gamma Q_rB)^{-1} (rC-\gamma Q_r(rA))\\
&=\alpha Q_rA+\left(\beta-\alpha Q_rB\right)(R_0-\gamma Q_rB)^{-1} (C-\gamma Q_rA)
\end{align*}
and such that
\begin{align*}
A_{\circ,r}& = A_r-B(R_0-\gamma Q_rB)^{-1}(C_r-\gamma Q_rA_r)\\
&=rA-B(R_0-\gamma Q_rB)^{-1}(rC-\gamma Q_r(rA))=r A_\circ, \\
\alpha_{\circ,r}&=\alpha_r-\left(\beta_r-\alpha_r Q_rB\right)(R_0-\gamma Q_rB)^{-1}\gamma\\
&=\frac{\alpha}{r}-\left(\frac{\beta}{r}-\frac{\alpha}{r}Q_rB\right)(R_0-\gamma Q_rB)^{-1} \gamma
=\frac{\alpha_\circ}{r}
\end{align*}
are both stable, corresponding to the claim of item (ii). Moreover, the matrix $Q_r$ with these properties is unique.
 It follows that the Riccati equation that $Q_r$ solved is independent of $r$, but it is less straightforward that the condition of having $r A_\circ$ and $\frac{\alpha_\circ}{r}$ stable does not introduce a dependency on $r$; in particular, $A_\circ$ and $\alpha_\circ$ in the above formulas may depend on $r$. To see that this is not the case, we note that the matrix $Q_r$ can be obtained as the limit of a Riccati difference equation associated with the finite section method for $T_{\Om_r}$, as discussed in Section 4 of \cite{FKR10}. Indeed, since $\Om_r$ is continuous on $\BT$ and we assume $T_{\Om_r}$ to be invertible, for $N$ large enough the $N$-th section of $T_{\Om,r}$, i.e., the Toeplitz block matrix
\begin{equation}\label{ToepMatNr}
T_{\Om_r,N}=\mat{R_{0,r}&R_{-1,r}&\cdots&R_{1-N,r}\\ R_{1,r}&R_{0,r}&\cdots&R_{2-N,r}\\ \vdots&\vdots&\ddots&\vdots\\ R_{N-1,r}&R_{N-2,r}&\cdots&R_{0,r}},
\end{equation}
with $R_{j,r}$ the $j$-th Fourier coefficient of $\Om_r$, will be invertible, and the matrices
\[
Q_{N,r}:= \cC_{\be_r,\al_r,N} T_{\Om_r,N}^{-1} \cO_{C_r,A_r,N}
\]
with $\cC_{\be_r,\al_r,N}$ and $\cO_{C_r,A_r,N}$ given by
\begin{equation}\label{ObsConrN}
\cC_{\be_r,\al_r,N}=\Row_{j=0}^{N-1}(\al_r^{j}\be_r)
\ands \cO_{C_r,A_r,N}=\Col_{j=0}^{N-1}(C_rA_r^{j})
\end{equation}
solve the Riccati difference equation
\begin{align*}
Q_{N+1,r}&= \alpha_r Q_{N,r} A_r + (\beta_r-\alpha_r Q_{N,r}B)(R_0-\gamma Q_{N,r}B)^{-1}(C_r-\gamma Q_{N,r} A_r)\\
   &= \alpha Q_{N,r} A + (\beta-\alpha Q_{N,r}B)(R_0-\gamma Q_{N,r}B)^{-1}(C-\gamma Q_{N,r} A)
\end{align*}
and $Q_{N,r}$ converges to $Q_r$ as $N\to \infty$. Hence, in order to see that $Q_r$ is independent of $r$, it suffices to show that $Q_{N,r}$ is independent of $r$. Note that the Fourier coefficients of $\Om_r$ are given by
\begin{equation}\label{RNr}
R_{n,r}=\left\{ \begin{array}{l}
C_rA_r^{n-1}B=r^n CA^{n-1}B\mbox{ for $n>0$},\\
R_0\mbox{ for $n=0$},\\
\gamma \alpha_r^{n-1}\beta_r=r^{-n} \gamma \alpha^{n-1}\beta\mbox{ for $n<0$}.
\end{array}\right.
\end{equation}
It follows that
\[
T_{\Om_r,N}=\diag(I_m,r I_m,\ldots,r^{N-1}I_m)T_{\Om,N} \diag(I_m,rI_m,\ldots,r^{N-1}I_m)^{-1},
\]
where $T_{\Om,N}$ is as in \eqref{ToepMatNr} with $R_{n,r}=R_{n,1}$, where $R_{n,1}$ is defined according to \eqref{RNr} with $r=1$. This shows that for large $N$, also $T_{\Om,N}$ invertible and
\begin{equation}\label{TOmrNinv}
T_{\Om_r,N}^{-1}=\diag(I_m,r I_m,\ldots,r^{N-1}I_m)T_{\Om,N}^{-1} \diag(I_m,rI_m,\ldots,r^{N-1}I_m)^{-1}.
\end{equation}
Define $\cC_{\be,\al,N}$ and $\cO_{C,A,N}$ analogous to $\cC_{\be_r,\al_r,N}$ and $\cO_{C_r,A_r,N}$, with $\be_r,\al_r,C_r,A_r$ replaced by $\be,\al,C,A$, respectively. It is then easy to see that
\begin{align*}
\cC_{\be_r,\al_r,N}&=\cC_{\be,\al,N}\diag(I_m,r I_m,\ldots,r^{N-1}I_m)^{-1},\\
\cO_{C_r,A_r,N}&= \diag(I_m,r I_m,\ldots,r^{N-1}I_m)\cO_{C,A,N}.
\end{align*}
Combining these identities with \eqref{TOmrNinv}, it follows that
\[
Q_{N,r}=\cC_{\be_r,\al_r,N} T_{\Om_r,N}^{-1} \cO_{C_r,A_r,N}=\cC_{\be,\al,N} T_{\Om,N}^{-1} \cO_{C,A,N},
\]
is indeed independent of $r$, and consequently, $Q_r$ is also independent of $r$.

It remains to prove the equivalence of (ii) (or (i)) and (iii),and that (iii) can be achieved as described in the proposition, and the formulas for $T_{\Om_r}^{-1}$. The equivalence of (ii) and (iii), in fact, follows directly from Theorem 1.1 in \cite{FKR10}. We now show that the formulas for the canonical factors from \cite{FKR10} lead to the factorization of $\Om_r$ using $\Theta$ and $\Psi$ from Theorem \ref{T:Main1}.  By the formulas in \cite{FKR10}, the canonical factorization of $\Om_r$ is given by $\Om_r(z)=\Psi^{(r)}(z)\Theta^{(r)}(z)$ where we factor $R_0-\ga Q_rB=\de D$, as claimed, and set
\begin{align*}
\be_{\circ,r}&=(\be_r-\al_r Q_rB)D^{-1}=r^{-1}(\be-\al Q_rB)D^{-1}=r^{-1}\be_\circ,\\
C_{\circ,r} &= \de^{-1}(C_r-\ga Q A_r)=r \de^{-1}(C-\ga Q A)=r C_\circ,
\end{align*}
to arrive at
\begin{align*}
\Psi^{(r)}(z) &=\de+\ga(zI-\al_r)^{-1}\be_{\circ,r} = \de+r^{-1}\ga(zI-r^{-1}\al)^{-1}\be_{\circ}\\
&=\de+\ga(rzI-\al)^{-1}\be_{\circ}=\Psi_r(z)
\end{align*}
with inverse
\begin{align*}
\Psi^{(r)}(z)^{-1} &=\de^{-1}-\de^{-1}\ga(zI-\al_{\circ,r})^{-1}\be_{\circ,r}\de^{-1}\\ &=\de^{-1}-r^{-1}\de^{-1}\ga(zI-r^{-1}\al_{\circ})^{-1}\be_{\circ}\de^{-1} =\de^{-1}-\de^{-1}\ga(rzI-\al_{\circ})^{-1}\be_{\circ}\de^{-1}
\end{align*}
and
\begin{align*}
\Theta^{(r)}(z) &=D+zC_{\circ,r}(I-zA_r)^{-1}B = D+rzC_{\circ}(I-rzA)^{-1}B=\Theta_r(z)
\end{align*}
with inverse
\begin{align*}
\Theta^{(r)}(z)^{-1} &=D^{-1}-zD^{-1}C_{\circ,r}(I-zA_{\circ,r})^{-1}BD^{-1}\\
&= D^{-1}-rzD^{-1}C_{\circ}(I-rzA_\circ)^{-1}BD^{-1}.
\end{align*}
The formula for $T_{\Om_r}^{-1}$ now follows simply from the text preceding Theorem 1.1 in \cite{FKR10}.
\end{proof}

Using the equivalence of (ii) and (iii) in Proposition \ref{P:OmrThm} it is easy to prove our second main result.

\begin{proof}[\bf Proof of Theorem \ref{T:Main1}]
Since $\Om_r$, $\Psi_r$ and $\Theta_r$ are rational matrix functions, the factorization $\Om_r(z)=\Psi_r(z)\Theta_r(z)$ for some $1<r<r_0$ implies that also $\Om(z)=\Psi(z)\Theta(z)$ as well as the formulas for the inverses of $\Psi$ and $\Theta$. Hence (iii) in Proposition \ref{P:OmrThm} is equivalent to (ii) in Theorem \ref{T:Main1}.

Next we show that the realizations of $\Theta$ and $\Psi$ in \eqref{ThetaPsi} and of $\Theta^{-1}$ and $\Psi^{-1}$ in \eqref{ThetaPsiInvReal} are minimal. Note that since the realization of $\Om$ is minimal and $A$ is stable and $\al$ is semi-stable, the McMillan degree, $\degr(\Om)$, of $\Om$ is equal to the sum of the sizes of $A$ and $\al$, say $s$ and $t$, respectively. From the formulas of $\Theta$ and $\Psi$ it is clear that $\degr(\Theta)\leq s$ and $\degr(\Psi)\leq t$. On the other hand, since the McMillan degree is sublogarithmic, we have
\[
s+t=\degr(\Om)=\degr(\Psi \Theta)\leq \degr(\Psi) + \degr(\Theta)\leq s+t.
\]
Hence we have equality in each step, which implies $\degr(\Theta)=s$ and $\degr(\Psi)=t$, in other words, the realizations of $\Theta$ and $\Psi$ are minimal. By the observation right after Proposition 7.2 in \cite{BGKR08}, it follows that the realization for $\Theta$ (respectively $\Psi$) is minimal if and only if the realization of $\Theta^{-1}$ (respectively $\Psi^{-1}$) is minimal. Hence, also the realizations of $\Theta^{-1}$ and $\Psi^{-1}$ are minimal.

Since the solution $Q=Q_r$ of \eqref{RiccOmr} used to construct $r A_\circ$ and $r^{-1}\al_\circ$ is independent of $r$, it follows that the solution $Q$ in item (ii) in Proposition \ref{P:OmrThm} is such that $r A_\circ$ and $r^{-1}\al_\circ$ are stable for all $1<r<r_0$, so that $A_\circ$ is stable and $\al_\circ$ is semi-stable.

Thus, from the equivalence of (ii) and (iii) in Proposition \ref{P:OmrThm} it follows that we get the equivalence of (i) and (ii) in Theorem \ref{T:Main1}, except that at this stage we only get $\al_\circ$ to be semi-stable. To see that $\al_\circ$ is in fact stable, note that $\Om(z)=\Psi(z)\Theta(z)$ implies that $\Psi(z)^{-1}=\Theta(z)\Om(z)^{-1}$. Since $A$ is stable, $\Theta$ has no poles on $\BT$, and $\Om^{-1}$ has no poles on $\BT$ because $\Om$ is assumed to have no zeroes on $\BT$. Therefore, $\Psi^{-1}$ has no poles on $\BT$ either, which implies, by minimality of the realization of $\Psi^{-1}$, that $\al_\circ$ has no eigenvalues on $\BT$. Hence $\al_\circ$ is stable.
\end{proof}

\section{Fredholmness of $T_\Om$ versus Fredholmness of $T_{\Om_r}$}\label{S:FredholmCompare}

In this section we prove Theorem \ref{T:Main2}. The proof relies heavily on the connection between $T_\Om$ and $T_{\Om_r}$, with $\Om_r$ as in \eqref{Omr}, as explained in the next result.

\begin{proposition}\label{P:TomVsTomr}
Let $\Om\in\Rat^{m \times m}$ with $\det \Om \not\equiv 0$ and assume that $\Om$ has no zeroes on $\BT$.  Let $r_0>1$ be such that $\Om$ has no zeroes in the annulus $\fA_{r_0}$ and no poles in $\fA_{r_0}\setminus\BT$. Define
\begin{equation}\label{cD_r}
\cD_r:=\{f\in H^p_m \colon f_r\in H^p_m\}.
\end{equation}
Then for each $1<r<r_0$ we have
\begin{equation}\label{incl}
\kernel T_\Om \subset \cD_r \subset H_m(\overline{\BD}) \subset \Dom (T_\Om).
\end{equation}
Moreover, $T_\Om$ maps $\cD_r$ into $\cD_r$, the inverse image $T_\Om^{-1}(\cD_r)$ of $\cD_r$ under $T_\Om$ lies in $\cD_r$, and we have
\begin{equation}\label{TomVsTomr}
(T_\Om f)_r=T_{\Om_r}f_r,\quad f\in\cD_r.
\end{equation}
\end{proposition}

\begin{proof}[\bf Proof]
Let $1<r<r_0$. We start by proving \eqref{incl}, except for the first inclusion.

The second inclusion is trivial. If $f\in\cD_r$, then $f$ has an analytic extension to $r\BD$, so that, in particular, $f$ is analytic on $\overline{\BD}$.

The argument to show that $H_m(\overline{\BD}) \subset \Dom (T_\Om)$ is similar to that in the scalar case \cite[Theorem 6.2]{GtHJ19b}. Note that each $f\in H_m(\overline{\BD})$ is also in $\cD_{r'}$ for some $1<r'$ sufficiently close to $1$. Hence it suffices to show that $\cD_r \subset \Dom (T_\Om)$. Let $f\in\cD_r$. Since $\Om$ is rational, $\Om f$ is meromorphic on $r\BD$ with finitely many poles, each of finite multiplicity. Computing the residues of the poles of each of the entries in the vector function $\Om f$ it is easy to write $\Om f$ in the form $g+\rho$ with $g\in L^p_m$ and $\rho\in\Rat^m_0(\BT)$, showing that $f\in \Dom(T_\Om)$. Hence, we have proved the second inclusion.

Next we show that $T_\Om$ maps $\cD_r$ into $\cD_r$ and that \eqref{TomVsTomr} holds. Let $f\in\cD_r$. Hence $f$ has an analytic extension to $r\BD$ and $|f(z)|^p$ is integrable on $r\BT$, that is, $f\in H^p_m(r\BT)$. By \eqref{incl}, $f\in\Dom(T_\Om)$ and hence $\Om f=g+\rho$ for some $g\in L^p_m$ and $\rho\in\Rat_0^m(\BT)$. Write $g=g_+ + g_-$ with $g_+\in H^p_m$ and $g_-\in K^p_m$, so that $T_\Om f=g_+$. Since $f$ is analytic on $r \BD$ and $\Om$ and $\rho$ are rational with no poles in $\fA_r\setminus\BT$, it follows that $g=\Om f - \rho$ must be analytic in $\fA_r\setminus\BT$ with the poles on $\BT$ all having finite multiplicity. However, $g\in L_m^p$, and hence cannot have poles of finite multiplicity on $\BT$. Thus $g$ is analytic in $\fA_r$. Hence, $g_+$ is analytic on $r\BD$ and $g_-$ on $\BC\setminus \overline{r^{-1}\BD}$. Since $r<r_0,$ by an argument similar to that in the first part of the proof, it follows that $g_+\in H^p_m(r\BT)$ and $g_-\in K^p_m(r^{-1}\BT)$. This implies that $g_{+,r}(z)=g_+(rz)$ and $g_{-,r}(z)=g_-(rz)$ define functions in $H^p_m$ and $K^{p}_m$, respectively. In particular, $g_+\in\cD_r$ and it follows that $T_\Om$ maps $\cD_r$ into $\cD_r$.
Moreover, $\rho$ is a rational matrix function with poles only in $\BT$, so that $\rho_r$ only has poles inside $\BD$ and we obtain that $\rho_r\in K^p_m$. Note further that on $\BT$
\begin{align*}
\Om_r(z)f_r(z)&=\Om(rz)f(rz)=g(rz)+\rho(rz)=g_r(z)+\rho_r(z)\\
&=g_{+,r}(z)+g_{-,r}(z)+\rho_r(z).
\end{align*}
Therefore, we have
\[
T_{\Om_r}f_r=\BP(\Om_r f_r)=\BP(g_{+,r}+g_{-,r}+\rho_r)=g_{+,r}=(T_\Om f)_r.
\]

Finally, we show that $T_\Om^{-1}(\cD_r)\subset \cD_r$. Since $0\in \cD_r$, this proves in particular that $\kernel T_\Om \subset \cD_r$ and hence the first inclusion of \eqref{incl}. To prove the inclusion, we require the Wiener-Hopf type factorization from \cite[Theorems 1.1 and 1.2]{GtHJR21}, namely  $\Om$ can be factored as
\begin{equation}\label{WH}
\Om(z)=\Om_-(z)\Xi(z)\Om_+(z),\quad\mbox{with} \ \ \Xi(z)=z^{-k}\Om_\circ(z)P_0(z)
\end{equation}
for some integer $k\geq 0$, $\Om_+,\Om_\circ,\Om_-\in \Rat^{m \times m}$ and $P_0\in\cP^{m \times m}$ such that $\Om_-$ and $\Om_-^{-1}$ are both minus functions (i.e., no poles outside $\BD$), $\Om_+$ and $\Om_+^{-1}$ are both plus functions (i.e., no poles in $\overline{\BD}$), $\Om_\circ=\diag_{j=1}^m(\phi_j)$ with $\phi_j\in\Rat(\BT)$ having no zeroes and having roots only on $\BT$ (in \cite[Theorems 1.1]{GtHJR21}, $\phi_j\in\Rat$ can have zeroes on $\BT$, but this cannot occur since $\Om$ has no zeroes on $\BT$), and $P_0$ a lower triangular polynomial with $\det(P_0(z))=z^N$ for some integer $N\geq 0$. It then follows from Theorem 1.3 in \cite{GtHJR21} that
\[
T_{\Om}=T_{\Om_-}T_\Xi T_{\Om_+} \quad\mbox{and}\quad
T_{\Om_-}^{-1}=T_{\Om_-^{-1}},\ \ T_{\Om_+}^{-1}=T_{\Om_+^{-1}}.
\]
To show that $T_\Om^{-1}(\cD_r)\subset \cD_r$, it suffices to show that $T_\Xi^{-1}(\cD_r)\subset \cD_r$, $T_{\Om_+}^{-1}(\cD_r)\subset \cD_r$ and $T_{\Om_-}^{-1}(\cD_r)\subset \cD_r$. The latter two inclusions follow from the the fact that $T_{\Om_-}$ and $T_{\Om_+}$ are invertible with inverses $T_{\Om_-}^{-1}=T_{\Om_-^{-1}}$ and $T_{\Om_+}^{-1}=T_{\Om_+^{-1}}$ along with the argument from the previous paragraph applied to $T_{\Om_-^{-1}}$ and $T_{\Om_+^{-1}}$ showing that $\cD_r$ is an invariant subspace for these two operators; for the latter, note that $\Om_-$ and $\Om_+$  do not have zeroes and poles on the annulus $\fA_r$. Hence it remains to show that $T_\Xi^{-1}(\cD_r)\subset \cD_r$.

Let $f\in \Dom(T_\Xi)$ such that $T_\Xi f\in \cD_r$. Since $f\in \Dom(T_\Xi)$, we can apply Lemma 2.1 from \cite{GtHJRsub} to conclude that $\Xi(z) f(z) =z^{-k} h(z)  + \eta(z)$ with $h\in H^p_m$ and $\eta=(\eta_1,\ldots,\eta_m)\in \Rat_0^m(\BT)$ of the form $\eta_j=r_j/q_j\in \Rat_0(\BT)$ with $q_j$ the denominator of the $j$-th diagonal element of $\Om_0$. Write $z^{-k} h(z) = h_-(z) + h_+(z)$ with $h_+\in H^p_m$ and $h_-\in K^p_m$. It is clear that $h_-$ is analytic on $\BC\setminus \{0\}$. Moreover, $h_+=T_\Xi f$, so that $h_+\in\cD_r$, by assumption. Since $\det \Xi\not\equiv 0$, we have $f=\Xi^{-1}h_+ + \Xi^{-1}h_- +\Xi^{-1}\eta$. Note that $\Xi^{-1}$ has no poles on the annulus $\fA_r$ and no zeroes on $\fA_r\setminus\BT$. Since $\Xi^{-1}$, $f$, $h_+$ and $h_-$ don't have poles on $\BT$, neither can $\Xi^{-1}\eta$. It follows that $\Xi^{-1}\eta$, $\Xi^{-1}h_+$ and $\Xi^{-1}h_-$ are all analytic on $\fA_r$. Therefore, $f$ is analytic on $\fA_r$. However, $f\in H^p_m$, so that $f$ in fact is analytic on $r\BT$. Using that $\Xi^{-1}$ is rational, $g_+\in\cD_r$, $g_-$ and $\rho$ are continuous on $r\BT$ it follows that $f$ is $p$-integrable on $r\BT$, and hence $f\in \cD_r$.
\end{proof}

\begin{lemma}\label{L:Fredholm}
Let $\Om\in\Rat^{m \times m}$ with $\det \Om \not\equiv 0$. Then $T_\Om$ is Fredholm if and only if $\Om$ has no zeroes on $\BT$. Assume $T_\Om$ is Fredholm and let $r_0>1$ be such that $\Om$ has no zeroes in the annulus $\fA_{r_0}$ and no poles in $\fA_{r_0}\setminus\BT$. Then $T_{\Om_r}$ is bounded and Fredholm for each $1<r<r_0$.
\end{lemma}

\begin{proof}[\bf Proof]
For $1<r<r_0$, by definition of $r_0$ it is clear that $T_{\Om_r}$ has no poles and no zeroes on $\BT$, so that $T_{\Om_r}$ is bounded and Fredholm. For $T_{\Om}$ the result is not included in \cite{GtHJR21} but follows from the results proved there.
Indeed, consider a Wiener-Hopf type decomposition of $\Om$ as in the proof of Proposition \ref{P:TomVsTomr}, e.g., as in \eqref{WH}. Since $\Om_+^{-1}$ is a plus function and $\Om_-^{-1}$ is a minus function, they do not have poles on $\BT$. Also, since $\det P_0(z)=z^N$ for some integer $N\geq 0$, $P_{0}^{-1}$ as a function in $\Rat^{m\times m}$ can only have a pole at 0, so that $P_0$ also does not have zeroes on $\BT$. This shows that the zeroes of $\Om$ on $\BT$ correspond to the zeroes of $\Om_\circ$. Since $\Om_\circ$ is a diagonal matrix function, its zeroes correspond to the zeroes of its diagonal elements, which are all on $\BT$, by construction. Since $\Om$ has no zeroes on $\BT$, this implies that the numerators of the diagonal elements of $\Om_\circ$ are constant, assuming the numerators and denominators are co-prime. The statement for $T_\Om$ now follows from the fact that $T_\Om$ is Fredholm if and only if the numerators (assuming co-primeness) in $\Om_\circ$ are constant, according to Theorem 1.4 in \cite{GtHJR21}.
\end{proof}

As a consequence of Proposition \ref{P:TomVsTomr}, it is easy to show that the dimensions of the kernels of $T_\Om$ and $T_{\Om_r}$ are the same.

\begin{corollary}\label{C:Kernel}
Let $\Om\in\Rat^{m \times m}$ with $\det \Om \not\equiv 0$ and assume that $\Om$ has no zeroes on $\BT$.  Let $r_0>1$ be such that $\Om$ has no zeroes in the annulus $\fA_{r_0}$ and no poles in $\fA_{r_0}\setminus\BT$. For each $1<r<r_0$ we have
\[
\kernel T_\Om=\{f_{1/r} \colon f\in \kernel T_{\Om_r}\}\ands
\kernel T_{\Om_r}=\{f_{r} \colon f\in \kernel T_{\Om}\}.
\]
In particular, we have $\dim \kernel T_\Om= \dim \kernel T_{\Om_r}$.
\end{corollary}

\begin{proof}[\bf Proof]
Note that the formula for $\kernel T_{\Om_r}$ makes sense, since $\kernel T_\Om\subset \cD_r$. Since the map $f\mapsto f_r$ defines a bijection from $\cD_r$ onto $H^p_m$, with inverse map $h\mapsto h_{1/r}$, it suffices to prove one of the two formulas. From Proposition \ref{P:TomVsTomr} it follows that
\[
f\in \kernel T_\Om\subset\cD_r \ \Longleftrightarrow \
0= T_\Om f \ \Longleftrightarrow \ 0=(T_\Om f)_r=T_{\Om_r} f_r.
\]
This proves the formula for $\kernel T_{\Om_r}$.
\end{proof}

With a bit more work we can prove a similar result for the codimensions of the ranges of $T_\Om$ and $T_{\Om_r}$.

\begin{corollary}\label{C:RangeCompl}
Let $\Om\in\Rat^{m \times m}$ with $\det \Om \not\equiv 0$ and assume that $\Om$ has no zeroes on $\BT$.  Let $r_0>1$ be such that $\Om$ has no zeroes in the annulus $\fA_{r_0}$ and no poles in $\fA_{r_0}\setminus\BT$. Let $1<r<r_0$ and let $\cX$ be a complement of $\Ran T_{\Om_r}$ in $H^p_m$. Then
\[
\cX_{1/r}:=\{h_{1/r}\colon h\in\cX\}
\]
is a complement of $\Ran T_{\Om}$.
In particular, we have
\[
\codim \Ran T_\Om= \codim \Ran T_{\Om_r}.
\]
\end{corollary}

\begin{proof}[\bf Proof]
Note that by assumption $T_\Om$ and $T_{\Om_r}$ are both Fredholm. Hence they have closed ranges and $\cX$ is finite dimensional. By Proposition \ref{P:TomVsTomr} we have that
\[
\{g_r\colon g\in T_\Om(\cD_r)\}=\Ran T_{\Om_r},
\]
and thus
\[
T_\Om(\cD_r)=(\Ran T_{\Om_r})_{1/r}:=\{g_{1/r}\colon g\in \Ran T_{\Om_r}\}.
\]
Since $\Ran T_{\Om_r} + \cX$ is a direct sum, the same is true for $(\Ran T_{\Om_r})_{1/r} + \cX_{1/r}=T_\Om(\cD_r) + \cX_{1/r}$, since $h\in (\Ran T_{\Om_r})_{1/r} \cap \cX_{1/r}$ implies that $h\in\cD_r$ and $h_r\in \Ran T_{\Om_r} \cap\cX=\{0\}$, so that $h=0$. Also, since $\Ran T_{\Om_r} + \cX=H^p_m$, we have that $T_\Om(\cD_r) + \cX_{1/r}=\cD_r$. We claim that $\Ran T_\Om + \cX_{1/r}$ is also a direct sum. Indeed, let $h\in \Ran T_\Om  \cap \cX_{1/r}$. Since $h\in \Ran T_\Om$, we have $h=T_\Om f$ for some $f\in\Dom(T_\Om)$. Moreover, we have $h\in\cD_r$, because $\cX_{1/r}\subset \cD_r$. But then Proposition \ref{P:TomVsTomr} implies that also $f\in\cD_r$. Hence $h\in T_\Om(\cD_r) \cap \cX_{1/r}=\{0\}$. Finally, note that
\[
\cD_r=T_\Om(\cD_r) + \cX_{1/r} \subset \Ran T_\Om + \cX_{1/r} \subset H^p_m,
\]
and that $\Ran T_\Om + \cX_{1/r}$ is closed, since $\Ran T_\Om$ and $\cX_{1/r}$ are both closed and $\cX_{1/r}$ finite dimensional, using \cite[Proposition III.4.3]{C85}. The fact that $\cD_r$ is dense now implies that $\Ran T_\Om + \cX_{1/r} = H^p_m$ is a direct sum decomposition of $H^p_m$.
\end{proof}

\begin{proof}[\bf Proof of Theorem \ref{T:Main2}]
The claims follow directly by combining the results of Lemma \ref{L:Fredholm} and Corollaries \ref{C:Kernel} and \ref{C:RangeCompl}.
\end{proof}

\section{Invertibility of $T_\Om$ and the formula for $T_\Om^{-1}$}\label{S:TOm invert}

In this section we prove Theorem \ref{T:Main3} and Proposition \ref{P:Qform}. That invertibility of $T_\Om$ corresponds to the existence of a solution to the Riccati equation \eqref{Ricc} so that $A_\circ$ and $\al_\circ$ in \eqref{Acirc-alphacirc} are stable, and hence the pseudo-canonical factorization of $\Om$, follows easily from the results of the previous sections. To obtain the formula for $T_\Om$ and its block matrix representation in terms of the realization \eqref{realOmega} requires more work. We shall first present the operator factorization of $T_\Om$ corresponding to the pseudo-canonical factorization.

\begin{lemma}\label{L:matrep TOm}
Let $\Om\in\Rat^{m \times m}$ with $\det \Om \not\equiv 0$ and no zeroes on $\BT$ be given by the minimal realization \eqref{realOmega} with $A$ stable and $\al$ semi-stable. Assume that $\Om$ admits a pseudo-canonical factorization $\Om=\Psi\Theta$ as in item (ii) of Theorem \ref{T:Main1}, with $\Theta$ and $\Psi$ as in \eqref{ThetaPsi}. Then $T_\Om$ is given by
\[
T_\Om=T_\Psi T_\Th.
\]
In particular, $T_\Th$ is bounded and has a bounded inverse $T_{\Th}^{-1}=T_{\Th^{-1}}$. Moreover, $T_\Psi$ admits an upper triangular block Toeplitz matrix representation with respect to the standard basis of $H^p_m$ given by
\begin{equation}\label{TPsiMatRep}
\left[T_\Psi\right]_{i,j} = 0 \mbox{ if $j<i$},\quad
\left[T_\Psi\right]_{i,j} = \de \mbox{ if $j=i$},\quad
\left[T_\Psi\right]_{i,j} = \ga \al^{j-i-1}\be_0 \mbox{ if $j>i$}.
\end{equation}
\end{lemma}

\begin{proof}[\bf Proof]
Since $A$ and $A_\circ$ are both stable, it follows that $\Th$ and $\Th^{-1}$ have no poles inside the closed unit disc $\ov{\BD}$, that is, both are plus functions. It then follows from Lemma 6.1 in \cite{GtHJR21} (with $\Om=\Psi$ and $V=\Th$) that $T_\Om=T_\Psi T_\Th$. It is then also clear that $T_\Th$ is bounded with bounded inverse $T_{\Th}^{-1}=T_{\Th^{-1}}$.  Hence, it remains to determine the block matrix representation of $T_\Psi$. For this purpose, we compute $T_\Psi z^n$. That should produce a polynomial, and the coefficients of that polynomial, augmented with zeroes, give the $n$-th block column of the block matrix representation of $T_\Psi$. By successive applications of the formula
\[
(zI - \al)^{-1}=z^{-1}I + z^{-1} \al (zI - \al)^{-1}
\]
we see that
\begin{align*}
\Psi(z)z^n&
=\left(\de +\sum_{j=0}^{n-1} z^{-j-1}\ga \al^{j}\be_\circ + z^{-n}\ga \al^n (zI - \al)^{-1} \be_\circ\right)z^n\\
&=\delta z^n +\sum_{j=0}^{n-1} z^{n-j-1}\gamma\alpha^j\beta_\circ
+\gamma \alpha^n (z I-\alpha)^{-1}\beta_\circ.
\end{align*}
Now, because $\alpha$ has all its eigenvalues in the closed unit disc $\ov{\BD}$, the function
$\gamma \alpha^n (z I-\alpha)^{-1}\beta_\circ$ can be written as the sum of a function in $K^p_{m \times m}$ and a function in $\Rat_0^{m \times m}(\BT)$. Thus $T_\Psi z^n=\delta z^n +\sum_{j=0}^{n-1} z^{n-j-1}\gamma\alpha^j\beta_\circ$. This shows that the block matrix representation of $T_\Psi$ is indeed given by \eqref{TPsiMatRep}.
\end{proof}

\begin{proof}[\bf Proof of Theorem \ref{T:Main3}]
First assume $T_\Om$ is invertible. By Theorem \ref{T:Main2}, it follows that $T_{\Om_r}$ is invertible for all $1<r<r_0$, with $r_0$ as in Theorem \ref{T:Main2}. It then follows from Proposition \ref{P:OmrThm} that a solution $Q$ to the Riccati equation \eqref{Ricc} exists, and reasoning as in the proof of Theorem \ref{T:Main1} if follows that for this solution $Q$ the matrices $A_\circ$ and $\al_\circ$ in \eqref{Acirc-alphacirc} are stable.

Conversely, if the Riccati equation \eqref{Ricc} has a solution $Q$ such that $A_\circ$ and $\al_\circ$ are stable, then Theorem \ref{T:Main1} provides a pseudo-canonical factorization of $\Om$, and, due to the stability of $A_\circ$ and $\al_\circ$, this factorization extends to a canonical factorization of $\Om_r$ for $r>1$ small enough. It then follows from Proposition \ref{P:OmrThm} that $T_{\Om_r}$ is invertible, and, consequently, that $T_\Om$ is invertible, by Theorem \ref{T:Main2}.

From the factorization $T_\Om=T_\Psi T_\Theta$ and the boundedness and boundedly invertibility of $T_\Theta$, obtained in Lemma \ref{L:matrep TOm}, it follows that invertibility of $T_\Om$ corresponds to invertibility of $T_\Psi$ and, moreover, $T_\Om^{-1}=T_\Theta^{-1}T_\Psi^{-1}=T_{\Theta^{-1}}T_\Psi^{-1}$. Since $\Theta$, $\Theta^{-1}$ and $\Psi^{-1}$ have no poles on $\BT$, the Toeplitz operators $T_\Theta$, $T_{\Theta^{-1}}$ and $T_{\Psi^{-1}}$ are bounded and their block matrix representations are well understood. It remains to show that the block matrix representation of $T_\Psi^{-1}$ and $T_{\Psi^{-1}}$ are the same, as this would prove that these bounded operators coincide on the subspace of polynomials $\cP^m$ and equality would follow from their boundedness and the denseness of $\cP^m$ in $H^p_m$. Indeed, once it is proved that $T_\Psi^{-1}=T_{\Psi^{-1}}$, then the block matrix representation of $T_\Om^{-1}$ in \eqref{TOminvmat} follows directly from $T_\Om^{-1}=T_{\Theta^{-1}}T_{\Psi^{-1}}$ and the block matrix representations of $T_{\Theta^{-1}}$ and $T_{\Psi^{-1}}$. Hence, we need to show that the matrix entries with respect to the standard (block) basis of $H^p_m$ of $T_{\Psi^{-1}} T_{\Psi}$ and $T_{\Psi}T_{\Psi^{-1}}$ are $I_m$ on the diagonal and $0$ elsewhere.

The block matrix representation of $T_\Psi$, in terms of the realization \eqref{realOmega}, is given by \eqref{TPsiMatRep}. The realization formula of $\Psi^{-1}$ in item (ii) of Theorem \ref{T:Main1}, together with the stability of $\al_\circ$ shows that the block matrix representation of $T_{\Psi^{-1}}$ is an upper block triangular Toeplitz matrix that is determined by its first block row, which is given by
\[
\mat{
 \delta^{-1} & -\delta^{-1}\gamma\beta_\circ\delta^{-1} &  -\delta^{-1}\gamma\alpha_\circ\beta_\circ\delta^{-1} & -\delta^{-1}\gamma\alpha_\circ^2\beta_\circ\delta^{-1} & \cdots}.
\]
We first consider the block matrix representation of $T_{\Psi^{-1}}T_\Psi$. It is required to show that the $(i,j)$-th block entry $[T_{\Psi^{-1}}T_\Psi]_{ij}$ works out as
\[
[T_{\Psi^{-1}}T_\Psi]_{ij}=0 \mbox{ if $j<i$},\quad  [T_{\Psi^{-1}}T_\Psi]_{ij}=I_m \mbox{ if $j=i$}
,\quad  [T_{\Psi^{-1}}T_\Psi]_{ij}=0 \mbox{ if $j>i$}.
\]
The case where $j<i$ follows directly because the matrix representations of $T_{\Psi^{-1}}$ and $T_\Psi$ are both block upper triangular, and the case $j=i$ follows because the block diagonal elements are each others inverses. Hence, it remains to consider the case where $j>i$. For this purpose, notice that
\begin{equation}\label{AalIds}
\al - \al_\circ = \be_\circ \de^{-1} \ga \ands A- A_\circ = B D^{-1}C_\circ.
\end{equation}
For $j>i$ we have
\begin{align*}
[T_{\Psi^{-1}}T_\Psi]_{ij} & =-\de(\de^{-1} \ga \al_\circ^{j-i-1}\be_\circ\de^{-1})
- \sum_{k=0}^{j-i-2}  \ga \al^k \be_\circ \de^{-1} \ga \al_\circ^{j-i-2-k}\be_\circ \de^{-1} +\\
& \qquad \qquad+ \ga \al^{j-i-1}\be_\circ \de^{-1}.
\end{align*}
If $j=i+1$, then the summation in the middle term of the right-hand side is empty, and it is easy to see that the right-hand side collapses to 0 by a direct application of the first identity in \eqref{AalIds}. For $j>i+1$, using the first identity in \eqref{AalIds}, we see that
\begin{align}
& \sum_{k=0}^{j-i-2}  \ga \al^k \be_\circ \de^{-1} \ga \al_\circ^{j-i-2-k}\be_\circ \de^{-1}=
\sum_{k=0}^{j-i-2}  \ga \al^k (\al-\al_\circ) \al_\circ^{j-i-2-k}\be_\circ \de^{-1}= \notag\\
&\qquad \qquad =\sum_{k=0}^{j-i-2}  \ga \al^{k+1} \al_\circ^{j-i-2-k}\be_\circ \de^{-1} -\sum_{k=0}^{j-i-2}  \ga \al^k \al_\circ^{j-i-1-k}\be_\circ \de^{-1} \notag\\
&\qquad \qquad =\sum_{k=1}^{j-i-1}  \ga \al^{k} \al_\circ^{j-i-1-k}\be_\circ \de^{-1} -\sum_{k=0}^{j-i-2}  \ga \al^k \al_\circ^{j-i-1-k}\be_\circ \de^{-1}.\label{form1}
\end{align}
Inserting this formula back into the formula for $[T_{\Psi^{-1}}T_\Psi]_{ij}$, it follows that in the first summation in \eqref{form1} the term $k=0$ is added, while in the second summation the term $k=j-i-1$ is added, so that $[T_{\Psi^{-1}}T_\Psi]_{ij}=0$, as claimed. A similar computation, using the second identity of \eqref{AalIds}, shows that the block matrix representation of $T_\Psi T_{\Psi^{-1}}$ also corresponds to the block matrix representation of $I_{H^p_m}$.
\end{proof}

Finally, we turn to the proof of the last result in the introduction. Proposition \ref{P:Qform} is stated for $p=2$, but the lemma which we require for the proof also works for $p\neq 2$. For $r>1$, define the invertible linear map
\[
\Up :\cD_r \to H^p_m,\ \ \Up:f \mapsto f_r,\quad\mbox{with inverse}\quad \Up^{-1}: H^p_m \to H^p_m,\ \ \Up^{-1} :f \mapsto f_{r^{-1}}.
\]

\begin{lemma}\label{L:rtononr}
Let $\Om\in\Rat^{m \times m}$ with $\det \Om \not\equiv 0$ be given by the minimal realization \eqref{realOmega} with $A$ stable and $\al$ semi-stable. Define $r_0$ as in Theorem \ref{T:Main2}, $\cO_{C,A}$ as in \eqref{ObsCA} and $\cC_{\al,\be}$ as in \eqref{Conalbe}, and define $\cO_{C_r,A_r}$ and $\cC_{\al_r,\be_r}$ analogously, where $C_r,A_r,\al_r,\be_r$ are defined as in \eqref{OmrRealMatrices} and $1<r<r_0$. Then $\cO_{C,A}$, $\cO_{C_r,A_r}$ and $\cC_{\al_r,\be_r}$ are bounded, the range of $\cO_{C,A}$ is contained in $\cD_r$ and $\cD_r$ is contained in $\Dom(\cC_{\al,\be})$. Moreover, we have
\[
r\cC_{\al_r,\be_r}=\cC_{\al,\be}\Up^{-1},\quad \cO_{C_r,A_r}=r\Up\cO_{C,A} \ands T_{\Om_r}=\Up T_{\Om}\Up^{-1}.
\]
Furthermore, in case $T_{\Om}$ is invertible, then $T_{\Om_r}$ is invertible as well and $T_{\Om_r}^{-1}=\Up T_{\Om}^{-1}\Up^{-1}$.

\end{lemma}

\begin{proof}[\bf Proof]
Since the realization \eqref{realOmega} of $\Om$ is minimal, it follows from the definition of $r_0$ that $A_r=rA$ is still stable. Since $r>1$ and $\al$ is semi-stable, $\al_r=r^{-1}\al$ is stable. This implies the boundedness of $\cO_{C,A}$, $\cO_{C_r,A_r}$ and $\cC_{\al_r,\be_r}$, as well as the fact that $\cO_{C,A}$ maps into $\cD_r$. The identity $\cO_{C_r,A_r}=r\Up\cO_{C,A}$ is straightforward from the definitions.

To show that $\cD_r$ is in the domain of $\cC_{\al,\be}$ and that $r\cC_{\al_r,\be_r}=\cC_{\al,\be}\Up^{-1}$ holds on $H^p_m$, let $f(z)=\sum_{k=0}^\infty f_k z^k \in H^p_m$, then $f$ is integrable in $\BT$ and so is the rational matrix function $(zI-\al_r)^{-1}\be_r=\sum_{k=0}^\infty \al_r^k \be_r z^k$. It follows from Lemma 1.5 on page 81 of \cite{SS11} that
\begin{align*}
\cC_{\al_r,\be_r}f&= \sum_{k=0}^\infty \al_r^k \be_r f_k =  \sum_{k=0}^\infty r^{-k-1} \al^k \be f_k
=  r^{-1}\sum_{k=0}^\infty \al^k \be r^{-k} f_k\\
&= r^{-1} \cC_{\al,\be}f_{1/r}= r^{-1} \cC_{\al,\be} \Up^{-1}f.
\end{align*}
The above computation shows in particular that $\Up^{-1}f$ is in $\Dom(\cC_{\al,\be})$ for each $f\in H^p_m$ so that $\cD_r\subset \Dom(\cC_{\al,\be})$.

The relation between $T_\Om$ and $T_{\Om_r}$ in \eqref{TomVsTomr} implies that $\Up T_\Om|_{\cD_r}=T_{\Om_r}\Up$. Since the range of $\Up^{-1}$ is equal to $\cD_r$ and $\Up\Up^{-1}=I$, we have $T_{\Om_r}=\Up T_\Om \Up^{-1}$, as claimed.

Assume $T_{\Om}$ is invertible. By Theorem \ref{T:Main2}, also $T_{\Om_r}$ is invertible. By Proposition \ref{P:TomVsTomr}, $T_\Om$ maps $\cD_r$ into $\cD_r$ and $T_\Om^{-1}$ also maps $\cD_r$ into $\cD_r$. In particular, the operator $\Up T_\Om^{-1} \Up^{-1}$ is a well-defined linear map on $H^p_m$, which is closed by \cite[Problem III.5.7]{K80}, and hence bounded by the closed graph theorem. To see that $T_{\Om_r}^{-1}=\Up T_\Om^{-1} \Up^{-1}$, note that if $f\in H^p_m$, then $\Up^{-1}f$ and $T_\Om \Up^{-1}f$ are in $\cD_r$ so that
\begin{align*}
(\Up T_\Om^{-1} \Up^{-1})(\Up T_\Om \Up^{-1})f
&= \Up T_\Om^{-1} \Up^{-1}\Up (T_\Om \Up^{-1}f)\\
&= \Up T_\Om^{-1} T_\Om \Up^{-1}f = \Up  \Up^{-1}f = f.
\end{align*}
Hence $(\Up T_\Om^{-1} \Up^{-1})(\Up T_\Om \Up^{-1})=I$. Similarly one obtains $(\Up T_\Om \Up^{-1})(\Up T_\Om^{-1} \Up^{-1})\allowbreak=I$. Hence, $T_{\Om_r}^{-1}=\Up T_\Om^{-1} \Up^{-1}$.
\end{proof}

With the identities of the previous lemma, we can now prove Proposition \ref{P:Qform}.

\begin{proof}[\bf Proof of Proposition \ref{P:Qform}] Set $p=2$ and identify $H^2_m$ and $\ell^2_m$ in the usual way.
According to \cite{FKR10}, the matrix $Q$ in Proposition \ref{P:OmrThm} is given by $Q=\cC_{\al_r,\be_r}T_{\Om_r}^{-1}\cO_{C_r,A_r}$. It was already noted in Proposition \ref{P:OmrThm} that $Q$ is independent of the value of $1<r<r_0$. That also $Q=\cC_{\al,\be}T_{\Om}^{-1}\cO_{C,A}$ now follows directly from the identities derived in Lemma \ref{L:rtononr}.
\end{proof}

\subsection*{Acknowledgements}

This work is based on research supported in part by the National Research Foundation of South Africa (NRF, Grant Numbers 118513, 127364 and 145688) and the DSI-NRF Centre of Excellence in Mathematical and Statistical Sciences (CoE-MaSS). Any opinion, finding and conclusion or recommendation expressed in this material is that of the authors and the NRF and CoE-MaSS do not accept any liability in this regard.

\medskip
\noindent
{\it Declaration of interest:} none.

\end{document}